\documentclass[11pt]{amsart}
\usepackage{amsmath, amssymb, amscd, mathrsfs, url, pinlabel,hyperref, verbatim}
\usepackage[margin=1.25in,centering,letterpaper,dvips]{geometry}
\usepackage{color,dcpic,latexsym,graphicx,epstopdf,comment}
\usepackage[all]{xy}
\usepackage{enumitem}

\newtheorem {theorem}{Theorem}
\newtheorem {lemma}[theorem]{Lemma}
\newtheorem {proposition}[theorem]{Proposition}
\newtheorem {corollary}[theorem]{Corollary}

\newtheorem {definition}[theorem]{Definition}

\theoremstyle{remark}

\newtheorem {remark}[theorem]{Remark}

\numberwithin{equation}{section}
\numberwithin{theorem}{section}

\newlist{pcases}{enumerate}{1}
\setlist[pcases]{
  label=\bf{Case~\arabic*:}\protect\thiscase.~,
  ref=\arabic*,
  align=left,
  labelsep=0pt,
  leftmargin=0pt,
  labelwidth=0pt,
  parsep=0pt
}
\newcommand{\case}[1][]{%
  \if\relax\detokenize{#1}\relax
    \def\thiscase{}%
  \else
    \def\thiscase{~#1}%
  \fi
  \item
}

\newcommand{\Z}{\mathbb{Z}}

\newcommand{\Q}{\mathbb{Q}}

\newcommand{\rpthree}{\mathbb{RP}^3}

\newcommand{\img}{\operatorname{Im}}

\newcommand{\ab}{\operatorname{ab}}
\newcommand{\can}{\operatorname{can}}

\title{Fillings of unit cotangent bundles}
\date{}

\begin{document}

\author{Steven Sivek}
\email{ssivek@math.princeton.edu}
\address{Department of Mathematics\\Princeton University}

\author{Jeremy Van Horn-Morris}
\email{jvhm@uark.edu}
\address{Department of Mathematical Sciences\\The University of Arkansas}

\begin{abstract}
We study the topology of exact and Stein fillings of the canonical contact structure on the unit cotangent bundle of a closed surface $\Sigma_g$, where $g$ is at least 2.  In particular, we prove a uniqueness theorem asserting that any Stein filling must be s-cobordant rel boundary to the disk cotangent bundle of $\Sigma_g$.  For exact fillings, we show that the rational homology agrees with that of the disk cotangent bundle, and that the integral homology takes on finitely many possible values: for example, if $g-1$ is square-free, then any exact filling has the same integral homology and intersection form as $DT^*\Sigma_g$.
\end{abstract}

\maketitle

\section{Introduction}

The unit cotangent bundle $ST^*M$ of a Riemannian manifold $M$ is equipped with a canonical contact structure $\xi_{\can}$, given in local coordinates as the kernel of $\alpha = \sum_i p_i dq_i$.  The contact manifold $(ST^*M, \xi_{\can})$ is Stein fillable, with one filling given by the disk cotangent bundle $DT^*M$, and it is natural to ask whether other such fillings exist.  Our goal in this paper is to study the Stein fillings, and more generally the exact symplectic fillings, of $(ST^*M,\xi_{\can})$ in the case where $M=\Sigma_g$ is a surface of genus $g \geq 2$.  We will denote the contact manifold $(ST^*\Sigma_g, \xi_{\can})$ by $(Y_g,\xi_g)$.

In the cases $g=0,1$ the fillings of $(Y_g,\xi_g)$ are already understood, and we know that in fact any minimal symplectic filling must be diffeomorphic to $DT^*\Sigma_g$.  McDuff \cite{mcduff-rational-ruled} proved this for $(ST^*S^2 = \rpthree,\xi_{\can})$, and then Hind \cite{hind-rp3} showed that $DT^*S^2$ is the unique Stein filling up to Stein homotopy.  Similarly, Stipsicz \cite{stipsicz} proved that a Stein filling of the unit cotangent bundle $ST^*T^2 = T^3$ must be homeomorphic to $DT^*T^2 \cong T^2 \times D^2$, and Wendl \cite{wendl} showed that all of its minimal strong symplectic fillings are symplectically deformation equivalent to $DT^*T^2$.

For $g \geq 2$, however, no such uniqueness results for symplectic fillings are possible.  This was observed by Li, Mak, and Yasui \cite[Proposition~3.3]{lmy}, who noted that in this case McDuff \cite{mcduff-disconnected} constructed a symplectic 4-manifold which strongly fills its disconnected boundary, one of whose components is $(Y_g,\xi_g)$.  One can glue a symplectic cap with $b_2^+$ arbitrarily large to the remaining component to get an arbitrarily large filling of $(Y_g,\xi_g)$; or, as pointed out by Wendl \cite{wendl-blog}, one can even use this to construct a strong symplectic cobordism from any contact 3-manifold to $(Y_g,\xi_g)$.

Despite this, if we require the fillings in question to be exact or Stein then the situation is drastically simpler.  Our main results are the following, which appear as Theorem~\ref{thm:homotopy-equivalent} and Theorem~\ref{thm:exact-homology} respectively.
\begin{theorem} \label{thm:main-stein}
If $(W,J)$ is a Stein filling of $(Y_g,\xi_g) = (ST^*\Sigma_g, \xi_{\can})$, then $W$ is s-cobordant rel boundary to the disk cotangent bundle $DT^*\Sigma_g$.
\end{theorem}
In particular, $W$ is homotopy equivalent rel boundary to $DT^*\Sigma_g$.

\begin{theorem} \label{thm:main-exact}
If $(W,\omega)$ is an exact symplectic filling of $(Y_g,\xi_g) = (ST^*\Sigma_g, \xi_{\can})$, then the homology of $W$ is given by 
\begin{align*}
H_1(W;\Z) &\cong \Z^{2g} \oplus \Z/d\Z, & H_2(W;\Z) &\cong \Z, & H_3(W;\Z) &= 0
\end{align*}
for some integer $d$ such that $d^2$ divides $g-1$, and the intersection form on $H_2(W)$ is $\langle \frac{2g-2}{d^2}\rangle$.
\end{theorem}

\begin{remark}
The requirement that $d^2 \mid g-1$ implies that the integral homology and intersection form of an exact filling of $(Y_g,\xi_g)$ are uniquely determined (hence isomorphic to those of $DT^*\Sigma_g$) whenever $g-1$ is square-free.  This condition is well-known to hold for a subset of the natural numbers with density $\frac{6}{\pi^2}$.
\end{remark}

\begin{remark}
Li, Mak, and Yasui have independently proved a stronger version of Theorem~\ref{thm:main-exact}, namely that every exact filling of $(Y_g,\xi_g)$ has the integral homology and intersection form of $DT^*\Sigma_g$, by similar arguments.  In other words, every exact filling has $d=1$.
\end{remark}

One notable feature of Theorems~\ref{thm:main-stein} and \ref{thm:main-exact} is that all of the fillings involved have $b_2^+(W)$ positive.  As far as we are aware, any classification theorems which have been proved to date for symplectic or Stein fillings of fillable contact 3-manifolds $(Y,\xi)$ have the feature that all of the symplectic fillings have $b_2^+=0$.  This is true because the classifications usually follow from one of two starting points: either $(Y,\xi)$ has a symplectic cap containing a symplectic sphere of nonnegative self-intersection, or $(Y,\xi)$ is supported by a planar open book.

In the first of these cases, it follows from McDuff \cite{mcduff-rational-ruled} that any filling embeds into a blow-up of either $\mathbb{CP}^2$ or a ruled surface, and in either case the closed manifold has $b_2^+=1$, with $H_2^+$ generated by the symplectic sphere inside the cap.  In the second case, the classifications use work of Wendl \cite{wendl}, who showed that all Stein fillings admit Lefschetz fibrations corresponding to factorizations of the monodromy into positive Dehn twists; but the planarity implies by a result of Etnyre \cite{etnyre} (whose proof relies on \cite{mcduff-rational-ruled}) that the filling is negative definite.  These techniques have been applied successfully to many contact structures on lens spaces \cite{mcduff-rational-ruled, lisca-lens, pvhm, kaloti}, links of simple singularities \cite{ohta-ono}, and Seifert fibered spaces \cite{starkston}, among others.

The reason we are able to succeed in the absence of either technique is the use of a Calabi-Yau cap, as defined and studied by Li, Mak, and Yasui \cite{lmy}.  We find a Lagrangian $\Sigma_g$ inside a K3 surface with simply connected complement, and a Weinstein tubular neighborhood of this Lagrangian is symplectomorphic to the disk cotangent bundle of $\Sigma_g$, so its complement is a symplectic cap for $(Y_g,\xi_g)$.  Gluing this cap to any filling produces a closed 4-manifold $X$ of symplectic Kodaira dimension zero, and the classification of the latter \cite{morgan-szabo,bauer,li-quaternionic} tells us that $X$ must be an integral homology K3.  In Section~\ref{sec:exact} we then deduce Theorem~\ref{thm:main-exact} from careful application of the Mayer-Vietoris sequence, and following this we use properties of Stein fillings in Section~\ref{sec:stein} to pin down the fundamental group and prove Theorem~\ref{thm:main-stein}.

Finally, we remark that we would like to strengthen Theorem~\ref{thm:main-stein} by showing that any Stein filling $W$ of $(Y_g,\xi_g)$ is homeomorphic to $DT^*\Sigma_g$, but for now this may be out of reach using our techniques.  This would require a proof of the topological s-cobordism theorem for s-cobordisms between 4-manifolds with fundamental group $\pi_1(\Sigma_g)$, which is currently only known when the fundamental group is ``good'' (see Freedman-Quinn \cite{freedman-quinn}), and it is an open question whether surface groups are good.

\subsection*{Acknowledgments}
This work began at the Princeton Low-Dimensional Topology Workshop 2015, and we thank the participants for contributing to a productive environment.  We thank Matt Day, John Etnyre, Dave Futer and Yo'av Rieck for helpful conversations.  We are especially grateful to Ian Agol for pointing out to us that Proposition~\ref{prop:central-extension} should be true and explaining how it should follow from the RFRS condition. SS was supported by NSF grant DMS-1506157. JVHM was supported in part by Simons Foundation grant No. 279342.

\section{Calabi-Yau caps}

In this section, we will construct and study a certain type of concave filling which was originally used by Li, Mak, and Yasui \cite{lmy} to bound the topology of Stein fillings of a given manifold.
\begin{definition}
Let $(Y,\xi)$ be a contact manifold.  A \emph{Calabi-Yau cap} for $(Y,\xi)$ is a symplectic manifold $(W,\omega)$ with concave boundary $(Y,\xi)$ and torsion first Chern class, such that there is a contact form $\alpha$ for $\xi$ and a Liouville vector field $X$ near $Y=\partial W$ satisfying $\alpha = \iota_X \omega|_Y$.
\end{definition}

In this section we will show that $(Y_g,\xi_g)$ admits a simply connected Calabi-Yau cap by finding an embedded Lagrangian $\Sigma_g$ of genus $g$ inside the elliptic surface $E(2)$, which is a K3 surface.  The cap $X_g$ is then the complement of a Weinstein tubular neighborhood of $\Sigma_g$, and it is Calabi-Yau because a K3 surface has trivial canonical class.  We remark that Li, Mak, and Yasui construct a Calabi-Yau cap for $(Y_g,\xi_g)$ by finding a Lagrangian $\Sigma_g$ inside the standard symplectic $T^4$, but this larger cap enables us to place much stronger restrictions on the possible fillings.

\begin{theorem} \label{thm:K3-lagrangian}
The elliptic surface $E(2)$ contains a Lagrangian surface $\Sigma_g$ of genus $g$ such that the complement $X_g$ of a Weinstein tubular neighborhood of $\Sigma_g$ is simply connected.
\end{theorem}

\begin{proof}
We express the elliptic fibration $\pi: E(2) \to S^2$ as a fiber sum $E(1) \#_{T^2} E(1)$, where if $a$ and $b$ are a pair of curves in the torus $T^2$ which intersect exactly once then each fibration $E(1) \to S^2$ has six singular fibers with vanishing cycle $a$ and six with vanishing cycle $b$, corresponding to the relation $(ab)^6=1$ in the mapping class group of the torus.  We can think of the base of the fibration $\pi$ as a connected sum $S^2 = S^2 \# S^2$, with one copy of $E(1)$ over each summand.

Let $\gamma \subset S^2$ be a simple closed curve separating the two $S^2$ summands; then we can arrange for $\gamma$ to have a small collar neighborhood $A = (-\epsilon,\epsilon)\times \gamma \subset S^2$, with no critical values of $\pi$, so that the symplectic form on $\pi^{-1}(A) \cong A \times T^2$ is the product symplectic form induced by area forms on each factor.  In particular, if we pick distinct values $t_1,\dots,t_g \in (-\epsilon,\epsilon)$ then the $g$ disjoint tori $T_i = \{t_i\} \times \gamma \times a$ are all Lagrangian.

Now let $c \subset S^2$ be a matching path \cite{seidel} between two critical points, one in either $S^2$ summand of $S^2 = S^2 \# S^2$, which each have vanishing cycle $b$.  Then $c$ lifts to a Lagrangian sphere $S \subset E(2)$.  We can arrange for $c$ to intersect $\gamma$ transversely in a single point, and if each $t_i$ is sufficiently close to zero it follows that $S$ intersects each $T_i$ transversely in a single point as well, namely the point $a\cap b$ in the fiber above $c \cap (\{t_i\}\times\gamma)$.  We surger $S$ and $T_i$ together at each of these points \cite{lalonde-sikorav,polterovich} to produce a Lagrangian $\Sigma_g$ of genus $g$.  We now take $X_g = E(2) \smallsetminus N(\Sigma_g)$, where $N(\Sigma_g)$ is a small Weinstein neighborhood of $\Sigma_g$.

It remains to be seen that $X_g$ is simply connected.  Since $E(2)$ is simply connected, $\pi_1(X_g)$ is normally generated by the class of a meridian $\mu$ of the Lagrangian $\Sigma_g$.  We let $c' \subset S^2$ be a path in one of the two $S^2$ summands with endpoints at a pair of critical values which both have vanishing cycle $a$, such that $c'$ intersects $c$ once transversely and is disjoint from each of the $\{t_i\}\times \gamma$.  Then there is a sphere $S' \subset E(2)$ lying above $c'$ (which need not be Lagrangian) such that $S' \cap S$ is the single point $a \cap b$ in the fiber over $c' \cap c$, hence $S'$ intersects $\Sigma_g$ transversely in precisely this point.  We arrange for $S'$ to intersect $\overline{N(\Sigma_g)}$ in a single meridional disk $D$ about this point, and then $S' \cap X_g$ is a disk $\overline{S'\smallsetminus D}$ with boundary a meridian of $\Sigma_g$, so $[\mu]=0$ and we are done.
\end{proof}

\begin{proposition} \label{prop:cap-betti}
The cap $X_g$ has Betti numbers $b_2^+(X_g) = 2$ and $b_2^-(X_g)=19$.
\end{proposition}

\begin{proof}
Recall that $b_2^+(K3) = 3$ and $b_2^-(K3) = 19$.  If $D_g$ is the Weinstein neighborhood of the surface $\Sigma_g$, so that $D_g$ is symplectomorphic to the disk cotangent bundle of $\Sigma_g$, then $H_2(D_g) = H_2(\Sigma_g) = \Z$, and since $\Sigma_g$ has self-intersection $2g-2 > 0$ the signature of $D_g$ is 1.  By Novikov additivity we have $\sigma(D_g) + \sigma(X_g) = \sigma(K3) = -16$, so $X_g$ has signature $-17$.  It thus suffices to show that $b_2^+(X_g) = 2$.

Let $V_+ \subset H_2(K3;\Q)$ be a positive definite 3-dimensional subspace containing the class $[\Sigma_g]$.   We can extend $[\Sigma_g]$ to a rational basis of $V_+$ whose other two classes are orthogonal to $[\Sigma_g]$, and thus integral multiples of those two classes can be represented by surfaces which are disjoint from $\Sigma_g$ and even avoid the neighborhood $D_g$.  These surfaces span a positive definite subspace of $H_2(X_g;\Q)$, so that $b_2^+(X_g) \geq 2$.  However, if $b_2^+(X_g) \geq 3$ then adjoining $[\Sigma_g]$ to a basis of a 3-dimensional positive-definite subspace of $H_2(X_g;\Q)$ would yield $b_2^+(K3) \geq 4$, which is absurd.
\end{proof}

We can understand the homology of $X_g$ more precisely by considering the Mayer-Vietoris sequence associated to the decomposition $K3 = D_g \cup_{Y_g} X_g$.

\begin{proposition} \label{prop:cap-homology}
We have $H_1(X_g;\Z) = H_3(X_g;\Z) = 0$ and $H_2(X_g;\Z) \cong \Z^{21} \oplus \Z^{2g}$, where the intersection form on $X_g$ has block form $\left(\begin{smallmatrix}Q_{g}&0\\0&0\end{smallmatrix}\right)$ with respect to this decomposition for some nondegenerate form $Q_{g}$ on $\Z^{21}$.
\end{proposition}

\begin{proof}
The claim that $H_1(X_g) = 0$ follows immediately from $X_g$ being simply connected, so then $H^1(X_g)=0$ as well by the universal coefficient theorem.  For $H_3(X_g)$, we use the general fact that if $X$ is an orientable $n$-manifold with nonempty boundary, then $H_{n-1}(X)$ injects into $H_{n-1}(X,\partial X) \cong H^1(X)$: the long exact sequence of the pair $(X,\partial X)$ says that
\[ 0 \to H_{n}(X,\partial X) \to H_{n-1}(\partial X) \to H_{n-1}(X) \to H_{n-1}(X,\partial X) \]
is exact, and the map $H_n(X,\partial X) \to H_{n-1}(\partial X)$ is an isomorphism $\Z \xrightarrow{\sim} \Z$ since it carries the relative fundamental class $[X,\partial X]$ to $[\partial X]$.  In general, this implies that $H_{n-1}(X)$ is torsion-free since $H^1(X)$ is; in this case, since $H^1(X_g)=0$ we have $H_3(X_g)=0$ as well.

Since $H_1(K3) = H_3(K3) = 0$ and $H_1(X_g)=0$, a portion of the Mayer-Vietoris sequence is given by
\[ 0 \to H_2(Y_g) \xrightarrow{i_2} H_2(D_g) \oplus H_2(X_g) \xrightarrow{j} H_2(K3) \xrightarrow{\delta} H_1(Y_g) \xrightarrow{i_1} H_1(D_g) \to 0. \]
Now from $H_1(Y_g) \cong \Z^{2g} \oplus \Z/(2g-2)$ we compute $H_2(Y_g) \cong H^1(Y_g) \cong \Z^{2g}$, so that
\[ 0 \to \Z^{2g} \xrightarrow{i_2} \Z \oplus H_2(X_g) \xrightarrow{j} \Z^{22} \xrightarrow{\delta} \Z^{2g} \oplus \Z/(2g-2) \xrightarrow{i_1} \Z^{2g} \to 0 \]
is exact.  Any torsion in $H_2(X_g)$ must lie in $\ker(j) = \img(i_2)$, and since $\img(i_2) \cong \Z^{2g}$ is torsion-free it follows that $H_2(X_g)$ is as well, hence $H_2(X_g) = \Z^{b_2(X_g)}$.  The map $i_1: \Z^{2g} \oplus \Z/(2g-2) \to \Z^{2g}$ is surjective and its target $\Z^{2g}$ is free, so $\ker(i_1)$ must be precisely the torsion subgroup $\Z/(2g-2)$ of the domain.  Thus $\img(\delta) = \Z/(2g-2)$, and we deduce from the above sequence that
\[ 0 \to \Z^{2g} \xrightarrow{i_2} \Z \oplus \Z^{b_2(X_g)} \xrightarrow{j} \Z^{22} \to \Z/(2g-2) \to 0 \]
is exact.  It follows that $\img(j)$ has index $2g-2$ and that $b_2(X_g) = 21+2g$; the latter fact implies that $H_2(X_g) = \Z^{21+2g}$.

Next, we note that the image of the natural map $H_2(Y_g) \to H_2(D_g)$ contributes to $b_2^0(D_g)$, since any surface inside $Y_g$ can be displaced inside a collar neighborhood of $\partial D_g$; but $H_2(D_g)$ is torsion-free and positive definite, so the map $H_2(Y_g) \to H_2(D_g)$ must be zero.  Since $i_2$ is injective, it follows that $H_2(Y_g) = \Z^{2g}$ injects into $H_2(X_g) = \Z^{21+2g}$.  Then $\ker(j) \subset H_2(X_g)$ is isomorphic to $\Z^{2g}$, so that the map $H_2(X_g) \to H_2(K3)$ has rank 21.  Since its image is a subgroup of $\Z^{22}$ it is free abelian, hence isomorphic to $\Z^{21}$, and so we have an exact sequence
\[ 0 \to H_2(Y_g) \to H_2(X_g) \to \Z^{21} \to 0 \]
which splits because $\Z^{21}$ is free.  Thus we have a direct sum decomposition
\[ H_2(X_g) \cong \Z^{21} \oplus H_2(Y_g) \]
in which the $H_2(Y_g)$ summand lies in the kernel of the intersection form.  But we have seen that $b_2^+(X_g)+b_2^-(X_g) = 21$, so the intersection form must be nondegenerate on the $\Z^{21}$ summand and the proof is complete.
\end{proof}

\begin{remark} \label{rem:Xg-sublattice}
The kernel of $j: H_2(D_g) \oplus H_2(X_g) \to H_2(K3)$ is the $H_2(Y_g) \cong \Z^{2g}$ summand of $H_2(X_g)$, so it restricts to an injective map
\[ j': H_2(D_g) \oplus \Z^{21} \to H_2(K3) \]
whose domain is the lattice $\Z \oplus \Z^{21}$ with intersection form $\left(\begin{smallmatrix}2g-2&0\\0&Q_g\end{smallmatrix}\right)$ in block form, and $j'$ embeds this lattice as an index-$(2g-2)$ sublattice of $H_2(K3) \cong \Z^{22} \cong 3H \oplus -2E_8$.
\end{remark}

\section{The topology of exact fillings}
\label{sec:exact}

We can now use the Calabi-Yau cap $(X_g,\omega_g)$ provided by Theorem~\ref{thm:K3-lagrangian} to understand the topology of fillings of $(Y_g,\xi_g)$.

\begin{proposition}
Let $(W,\omega)$ be an exact symplectic filling of $(Y_g,\xi_g)$.  Then the closed symplectic manifold
\[ (Z,\omega_Z) = (W,\omega) \cup_{(Y_g,\xi_g)} (X_g,\omega_g) \]
is an integer homology K3, with $H_1(Z;\Z)=H_3(Z;\Z)=0$ and $H_2(Z;\Z) = \Z^{22}$.
\end{proposition}

\begin{proof}
We can easily verify that $K_Z\cdot[\omega_Z] = 0$, where $K_Z$ is the canonical class of $(Z,\omega_Z)$.  Indeed, we can express it as a sum $K_Z|_W \cdot [\omega] + K_Z|_{X_g} \cdot [\omega_g]$, and in the first term we have $[\omega]=0$ since the form $\omega$ is exact, while in the second term we have $K_Z|_{X_g} = 0$ because $X_g$ has trivial canonical class.  Moreover, we have $b_2^+(Z) \geq b_2^+(X_g) = 2$, with the latter equality provided by Proposition~\ref{prop:cap-betti}.

Since $b_2^+(Z) \geq 2$ and $K_Z \cdot [\omega_Z] = 0$, it follows from Taubes \cite{taubes-more} that the only Seiberg-Witten basic classes on $Z$ are $\pm K_Z$.  Further work of Taubes \cite{taubes-sw-gr} then shows that $K_Z = 0$, hence $0$ is the only basic class: indeed, $K_Z$ is Poincar\'e dual to an embedded symplectic surface $\Sigma$, and we have $K_Z \cdot [\omega_Z] = \int_\Sigma \omega_Z \geq 0$ with equality only if $[\Sigma] = 0$.  It follows that $Z$ must be symplectically minimal, since otherwise the blow-up formula \cite{fs-blowup} implies that there would be at least two basic classes.

We have now shown that $(Z,\omega_Z)$ is minimal with trivial canonical class, and this proves that its symplectic Kodaira dimension \cite{li-kodaira-zero} is zero.  By work of Morgan--Szab\'o \cite{morgan-szabo}, Bauer \cite{bauer}, and Li \cite{li-quaternionic}, it follows that $Z$ has the rational homology of a K3 surface, an Enriques surface, or a $T^2$-bundle over $T^2$.  The latter two cases imply $b_2(Z)=10$ and $b_2(Z) \leq 6$ respectively, and neither of these can happen -- we already know that $b_2(Z) \geq b_2^-(Z) \geq b_2^-(X_g) = 19$ -- so $Z$ is a rational homology K3.

Finally, if $H_1(Z)$ is nontrivial, then it is torsion and so the kernel of the abelianization map $\pi_1(Z) \to H_1(Z)$ has finite index in $\pi_1(Z)$.  If the corresponding finite cover $(\tilde{Z},\tilde{\omega}_Z) \to (Z,\omega_Z)$ has degree $n=|H_1(Z)|$, then $(\tilde{Z},\tilde{\omega}_Z)$ has symplectic Kodaira dimension zero and hence signature at least $-16$ \cite{bauer,li-quaternionic}, whereas $\sigma(\tilde{Z}) = -16n$, so we must have $n=1$ and thus $H_1(Z)=0$.  It follows from Poincar\'e duality and the universal coefficient theorem that $H_3(Z) \cong H^1(Z) = 0$, and that $H_2(Z) \cong H^2(Z)$ is torsion-free since $H_1(Z)=0$ is.
\end{proof}

\begin{remark} \label{rem:simply-connected}
In the above argument, we see that $Z$ has even intersection form since $K_Z = 0$ is a characteristic class.  If we can show that $\pi_1(Z) = 1$, then $Z$ will be a simply connected, even, smooth 4-manifold with $b_2^+(Z)=3$ and $b_2^-(Z)=19$, implying that it is homotopy equivalent and hence homeomorphic to a K3 surface \cite{freedman}.

For example, if $(W,J)$ is a Stein filling of $(Y_g,\xi_g)$ then the inclusion $Y_g \hookrightarrow W$ induces a surjection $\pi_1(Y_g) \to \pi_1(W)$, and the cap $X_g$ is simply connected, so van Kampen's theorem says that $\pi_1(Z) = \pi_1(W) \ast_{\pi_1(Y_g)} 1 = 1$.  Thus if $(W,J)$ is a Stein filling then $Z$ is homeomorphic to a K3 surface.
\end{remark}

\begin{corollary} \label{cor:exact-betti}
If $(W,\omega)$ is an exact symplectic filling of $(Y_g,\xi_g)$, then $W$ has the same Betti numbers as the disk cotangent bundle $DT^*\Sigma_g$, namely $b_3(W)=0$, $b_2^+(W)=1$ and $b_2^-(W) = b_2^0(W) = 0$, and $b_1(W)=2g$.
\end{corollary}

\begin{proof}
We glue the cap $(X_g,\omega_g)$ to $(W,\omega)$ to form $Z$, which is a homology K3 and thus has signature $-16$.  Novikov additivity says that $-16 = \sigma(W) + \sigma(X_g)$, and from Proposition~\ref{prop:cap-betti} we conclude that $\sigma(W) = 1$.  In particular, we have $b_2^+(W) \geq 1$.

Now we consider the Mayer-Vietoris sequence for $Z = W \cup_{Y_g} X_g$ with coefficients in $\Q$: since $b_2(Z)=22$ and $b_2(Y_g)=b_1(Y_g)=2g$, the part of the sequence between $H_3(Z;\Q)=0$ and $H_1(Z;\Q)=0$ has the form
\[ 0 \to \Q^{2g} \to \Q^{b_2(W)} \oplus \Q^{21+2g} \to \Q^{22} \to \Q^{2g} \to \Q^{b_1(W)} \oplus \Q^0 \to 0. \]
Since we already know that $b_2(W) \geq 1$, an easy exercise shows that in fact $b_2(W)=1$ and $b_1(W)=2g$; then $\sigma(W)=1$ implies that $b_2^+(W)=1$ and $b_2^-(W)=b_2^0(W)=0$ as claimed.  Similarly, between $H_4(W;\Q) \oplus H_4(X_g;\Q) = 0$ and $H_3(Z;\Q) = 0$, we have
\[ 0 \to \Q \to \Q \to \Q^{b_3(W)} \oplus \Q^0 \to 0 \]
and this can only be exact if $b_3(W) = 0$.
\end{proof}

\begin{theorem} \label{thm:exact-homology}
If $(W,\omega)$ is an exact filling of $(Y_g,\xi_g)$, then for some integer $d$ such that $d^2$ divides $g-1$ we have $H_3(W;\Z) = 0$; $H_2(W;\Z) \cong \Z$, with intersection form $\langle \frac{2g-2}{d^2}\rangle$; and $H_1(W;\Z) \cong \Z^{2g} \oplus \Z/d\Z$.
\end{theorem}

\begin{proof}
We recall from the proof of Proposition~\ref{prop:cap-homology} that $H_3(W)$ is torsion-free, and so $b_3(W)=0$ implies that $H_3(W)=0$.

Now we write $Z = W \cup_{Y_g} X_g$, with $Z$ an integer homology K3, and consider the Mayer-Vietoris sequence over $\Z$:
\[ 0 \to H_2(Y_g) \xrightarrow{i} H_2(W) \oplus H_2(X_g) \xrightarrow{j} H_2(Z) \xrightarrow{\delta} H_1(Y_g). \]
We know that $H_1(Y_g) = \Z^{2g} \oplus \Z/(2g-2)$, hence $H_2(Y_g) = H^1(Y_g) = \Z^{2g}$, and similarly we know the homology of $X_g$ from Proposition~\ref{prop:cap-homology}.  Any torsion in $H_2(W)$ must lie in $\ker(j) = \img(i)$ since $H_2(Z) = \Z^{22}$ is free, but $\img(i) \cong \Z^{2g}$ is also free, so $H_2(W)$ is torsion-free and thus $H_2(W) = \Z$ by Corollary~\ref{cor:exact-betti}. 

Since $H_2(W)$ is positive definite and both $H_2(Y_g)$ and $H_2(W)$ are torsion-free, the map $H_2(Y_g) \to H_2(W)$ is zero, and we know that $H_2(X_g)$ decomposes as $\Z^{21} \oplus H_2(Y_g)$.  Thus we can split off the $H_2(Y_g) \xrightarrow{\sim} H_2(Y_g)$ component of $i$ in the above sequence, leaving us with
\[ 0 \to \Z \oplus \Z^{21} \to H_2(Z) \xrightarrow{\delta} H_1(Y_g). \]
Let $\Lambda \subset H_2(Z)$ be the image of $\Z \oplus \Z^{21}$; then $\Lambda$ is a sublattice of rank 22, so it has finite index, which must equal $\lvert\img(\delta)\rvert$.  But from $H_1(Y_g) = \Z^{2g} \oplus \Z/(2g-2)$ it follows that $\img(\delta)$ is a subgroup of $\Z/(2g-2)$, which then has order $\frac{2g-2}{d}$ for some integer $d \geq 1$.  Since $H_1(X_g) = H_1(Z) = 0$, the portion $H_2(Z) \xrightarrow{\delta} H_1(Y_g) \to H_1(W) \to 0$ of the sequence shows that $H_1(W)$ is isomorphic to $H_1(Y_g) / \img(\delta) \cong \Z^{2g} \oplus \Z/d$.

Let $e_1,\dots,e_{22}$ be an integral basis of $\Lambda$, where $e_1$ generates the direct summand $H_2(W) \cong \Z$ and $e_2,\dots,e_{22}$ is an integral basis of $\Z^{21} \subset H_2(X_g)$, and form a matrix $A$ whose columns are the elements $e_1,\dots,e_{22}$ expressed in an integral basis of $H_2(Z) \cong \Z^{22}$.  Letting $Q_{Z}$ be the intersection form on $H_2(Z)$ in this latter basis, then, the intersection form on $\Lambda$ in the basis $\{e_i\}$ is given by $Q_\Lambda =A^{T}Q_{Z}A$, and we have $\det(Q_\Lambda) = \pm \left(\frac{2g-2}{d}\right)^2$ since $Q_{Z}$ is unimodular and $\lvert\det(A)\rvert = [H_2(Z):\Lambda] = \frac{2g-2}{d}$.

On the other hand, we can write $Q_\Lambda$ in block form with respect to this basis as $\left(\begin{smallmatrix}e_1\cdot e_1&0\\0&Q_g\end{smallmatrix}\right)$, where $Q_g$ is the nondegenerate intersection form on $\Z^{21} \subset H_2(X_g)$; note that $\lvert\det(Q_g)\rvert$ does not depend on $W$ but only on the cap $X_g$.  From this it is clear that $\det(Q_\Lambda) = (e_1\cdot e_1) \det(Q_g)$, so it follows that $(e_1\cdot e_1) \lvert\det(Q_g)\rvert = \left(\frac{2g-2}{d}\right)^2$.  In the case where $W$ is the disk cotangent bundle $DT^*\Sigma_g$ we have $e_1\cdot e_1 = \Sigma_g^2 = 2g-2$ and $d=1$ (see Remark~\ref{rem:Xg-sublattice}), so it follows that $\lvert\det(Q_g)\rvert$ is equal to $2g-2$.  We conclude that
\[ e_1\cdot e_1 = \frac{\left(\frac{2g-2}{d}\right)^2}{2g-2} = \frac{2g-2}{d^2}. \]
Since $Z$ is a homology K3 it has an even intersection form, so $e_1\cdot e_1$ must be an even integer and we have $d^2 \mid g-1$, completing the proof.
\end{proof}

\begin{corollary}
If $g-1$ is square-free then any exact filling of $(Y_g,\xi_g)$ has the same homology and intersection form as the disk cotangent bundle $DT^*\Sigma_g$.
\end{corollary}

\section{The topology of Stein fillings}
\label{sec:stein}

\subsection{The homology of a Stein filling}

In this section we further investigate the topology of a filling $(W,\omega)$ of $(Y_g,\xi_g)$ which is not only exact but Stein; in this case we denote it by $(W,J)$ to avoid confusion.  In this case $W$ has a handle decomposition consisting of only 0-, 1-, and 2-handles, from which it follows classically that the inclusion $i: Y_g \hookrightarrow W$ induces a surjection $\pi_1(Y_g) \xrightarrow{i_*} \pi_1(W)$.  We note that since $Y_g$ is a circle bundle over $\Sigma_g$ with Euler number $2g-2$, its fundamental group has presentation
\[ \pi_1(Y_g) = \left\langle a_1,\dots,a_g,b_1,\dots,b_g,t \mathrel{}\middle|\mathrel{} \prod_{i=1}^g [a_i,b_i] = t^{2g-2}, [a_i,t]=[b_i,t] = 1 \right\rangle, \]
where $t$ represents a circle fiber and is central.  We will define $2g+1$ distinguished elements of $\pi_1(W)$ by
\begin{align*}
\alpha_j &= i_*(a_j), & \beta_j &= i_*(b_j), & \tau &= i_*(t)
\end{align*}
for $j=1,\dots,g$.  Since $i_*$ is surjective, we know that $\tau$ is central and that these $2g+1$ elements generate $\pi_1(W)$; in fact, it turns out that $\alpha_1,\dots,\alpha_g$ and $\beta_1,\dots,\beta_g$ suffice.

\begin{proposition} \label{prop:stein-meridian}
Suppose that $(W,J)$ is a Stein filling of $(Y_g,\xi_g)$ as above, and let $H \subset \pi_1(Y_g)$ denote the subgroup generated by $a_1,\dots,a_g,b_1,\dots,b_g$.  If $i_*: \pi_1(Y_g) \to \pi_1(W)$ is the inclusion-induced map, then $i_*|_H$ is surjective; in other words, $i_*(H) = \pi_1(W)$, and so $\pi_1(W)$ is generated by the elements $\alpha_1,\dots,\alpha_g$ and $\beta_1,\dots,\beta_g$.
\end{proposition}

\begin{proof}
It is not hard to check that $H$ is normal of index $2g-2$, since the only other generator in the above presentation is central (namely $t$) and the quotient $\pi_1(Y_g)/H$ is $\langle t \mid t^{2g-2}=1\rangle$.  Since $i_*$ is surjective, it is also easy to see that $i_*(H)$ is a normal subgroup of $\pi_1(W)$.  Moreover, $i_*$ induces a map
\[ \pi_1(Y_g)/H \to \pi_1(W)/i_*(H) \]
between the respective quotients, and this map is surjective, so since $\pi_1(Y_g)/H$ is a finite cyclic group generated by $[t]$ it follows that $\pi_1(W)/i_*(H)$ is also a finite cyclic group which is generated by $[\tau]$.  Thus $i_*(H)$ is a normal subgroup of $\pi_1(W)$ of some finite index $k \geq 1$ which divides $|\pi_1(Y_g)/H| = 2g-2$.

Let $p:\tilde{W} \to W$ be a finite $k$-fold covering such that $p_*(\pi_1(\tilde{W})) = i_*(H)$.  Then $\tilde{W}$ is also a Stein domain, and its boundary $\tilde{Y} = \partial \tilde{W}$ is a $k$-fold cover of $Y_g = \partial W$, which must be connected since Stein domains have connected boundary.  Thus $G = (p|_{\tilde{Y}})_*(\pi_1(\tilde{Y}))$ is an index-$k$ subgroup of $\pi_1(Y_g)$.  The cover $\tilde{Y} \to Y_g$ is normal, implying that $G$ is moreover a normal subgroup of $\pi_1(Y_g)$: indeed, the cover $\tilde{W} \to W$ is normal since $i_*(H)$ is a normal subgroup of $\pi_1(W)$, so its deck transformations act transitively on each fiber, and these restrict to deck transformations of $\tilde{Y}$, so the latter act transitively on fibers of $\tilde{Y}$.

We now consider the commutative diagram
\[ \xymatrix{
\tilde{Y} \ar[r]^{\tilde{i}} \ar[d]_{p|_{\tilde{Y}}} & \tilde{W} \ar[d]^{p} \\
Y_g \ar[r]^i & W
} \]
where $i$ and $\tilde{i}$ are the respective inclusion maps of each manifold into the Stein domain which it bounds, and thus induce surjections on the respective fundamental groups.  We have $p_*(\pi_1(\tilde{W})) = i_*(H)$ by construction, and since $\tilde{i}_*(\pi_1(\tilde{Y})) = \pi_1(\tilde{W})$ we can write
\[ i_*(H) = p_*(\tilde{i}_*(\pi_1(\tilde{Y}))) = i_*((p|_{\tilde{Y}})_*(\pi_1(\tilde{Y}))) = i_*(G). \]
Thus if $t^j \in G$ for some $j$, then we have $\tau^j \in i_*(G) = i_*(H)$, and so $k$ divides $j$.

Now we consider the composition $\varphi: \tilde{Y} \xrightarrow{p} Y_g \to \Sigma_g$, where we are now using $p$ to denote the restriction $p|_{\tilde{Y}}$.  The preimage of a point $x \in \Sigma_g$ is a $k$-fold cover of the circle fiber above $x$ in $Y_g$, which we identify with $t$ (at least up to conjugation, since we should pick a base point).  If this preimage is disconnected, then one of its components is a circle $\gamma \subset \tilde{Y}$ which is an $l$-fold cover of the circle fiber in $Y_g$ for some $1 \leq l < k$.  Thus $p_*(\gamma)$ is conjugate to $t^l$; but $G$ is normal and does not contain $t^l$, so it cannot actually contain $p_*(\gamma)$ either.  We conclude that $\varphi^{-1}(x)$ is a circle, and hence that $\tilde{Y}$ is also a circle bundle over $\Sigma_g$.  Its Euler number is then $\frac{2g-2}{k}$, though we only need that it is nonzero: if it were zero, then the image under $p$ of a section would give a section of $Y_g$, which has nonzero Euler number.

From the above we see that $b_1(\tilde{Y}) = 2g$, and since $H_1(\tilde{Y})$ surjects onto $H_1(\tilde{W})$ it follows that $b_1(\tilde{W}) \leq 2g$, hence 
\[ \chi(\tilde{W}) = 1 - b_1(\tilde{W}) + b_2(\tilde{W}) \geq 1-2g. \]
But we also know that $\chi(\tilde{W}) = k\chi(W) = k(2-2g)$, so we have $k(2-2g) \geq 1-2g$, or equivalently $(k-1)(2-2g) \geq -1$.  Since $2-2g \leq -2$, this can only hold if $k=1$; but $k$ is the index of $i_*(H)$ in $\pi_1(W)$, so the two must be equal.
\end{proof}

\begin{theorem} \label{thm:homology-stein}
Let $(W,J)$ be a Stein filling of $(Y_g,\xi_g)$.  Then $W$ has the same integral homology and intersection form as the disk cotangent bundle $DT^*\Sigma_g$.  In particular, we have $H_1(W) \cong \Z^{2g}$, and the intersection form on $H_2(W) \cong \Z$ is $\langle 2g-2 \rangle$.
\end{theorem}

\begin{proof}
In light of Theorem~\ref{thm:exact-homology} we know that $H_1(W) \cong \Z^{2g} \oplus \Z/d\Z$ for some $d \geq 1$, and that it suffices to show that $d=1$.  Now according to Proposition \ref{prop:stein-meridian}, the fundamental group $\pi_1(W)$ is generated by the $2g$ elements $\alpha_1,\dots,\alpha_g,\beta_1,\dots,\beta_g$, hence its abelianization $H_1(W)$ is also generated by the corresponding homology classes.  However, if $d>1$ then any presentation of $\Z^{2g} \oplus \Z/d\Z$ requires at least $2g+1$ generators, so we must have $d=1$.
\end{proof}

Theorem~\ref{thm:homology-stein} tells us the first group homology of $\pi_1(W)$, since $H_1(\pi_1(W);\Z) = H_1(W;\Z)$.  The second homology of $\pi_1(W)$ will also be useful later:

\begin{proposition} \label{prop:stein-hurewicz}
If $(W,J)$ is a Stein filling of $(Y_g,\xi_g)$, then $H_2(\pi_1(W); \Z) \cong \Z$.
\end{proposition}

\begin{proof}
Let $\pi = \pi_1(W)$, and recall that $H_2(W) \cong \Z$.  The group homology $H_2(\pi;\Z) = H_2(K(\pi,1);\Z)$ is classically known to be isomorphic to the cokernel of the Hurewicz map
\[ h: \pi_2(W) \to H_2(W), \] 
which is $\Z / \img(h)$, so $H_2(\pi;\Z)$ is $\Z$ if the Hurewicz map is zero and finite otherwise.  The $2g$ elements $\alpha_1,\dots,\alpha_g,\beta_1,\dots,\beta_g$ generate $\pi$ by Proposition~\ref{prop:stein-meridian}, so their images generate $H_1(W;\Z) = \Z^{2g}$ and are thus linearly independent over $\Q$.

Supposing that $h$ is nonzero, we have $H_2(\pi;\Q) = 0$.  According to Stallings \cite[Theorem~7.4]{stallings-homology}, the linear independence of the $\alpha_i$ and $\beta_i$ in $H_1(\pi)$ and the vanishing of $H_2(\pi;\Q)$ guarantee that $\alpha_1,\dots,\alpha_g,\beta_1,\dots,\beta_g$ form a basis of a free subgroup of $\pi$, and we conclude that $\pi$ is the free group $F_{2g}$.  The element $\prod_{j=1}^g [\alpha_j,\beta_j]$ of $\pi$ is central, since it equals $\tau^{2g-2}$ and $\tau$ is central; but free groups have trivial center, so $\prod_{j=1}^g [\alpha_j,\beta_j] = 1$ and thus $\pi$ is a nontrivial quotient of $F_{2g}$.  Since finitely generated free groups are Hopfian, a nontrivial quotient of $F_{2g}$ cannot be isomorphic to $F_{2g}$, and we have a contradiction.
\end{proof}

Since the class $[t]$ of the circle fiber generates the torsion summand of $H_1(Y_g)$, and Theorem~\ref{thm:homology-stein} says that $H_1(W)$ is torsion-free, we see that $[\tau]=0$ in $H_1(W)$.  Thus $\tau$ lies in the commutator subgroup of $\pi_1(W)$.  In Section~\ref{ssec:pi_1} we will see that in fact $\tau=1$ in $\pi_1(W)$.

\subsection{The fundamental group of a Stein filling}
\label{ssec:pi_1}

Let $(W,J)$ denote a Stein filling of $(Y_g,\xi_g)$ as usual.  Our goal in this section is to explicitly determine its fundamental group:

\begin{theorem} \label{thm:surface-group}
The fundamental group $\pi_1(W)$ is isomorphic to $\pi_1(\Sigma_g)$.
\end{theorem}

Our strategy will be to first show that $\pi_1(W)$ must be an extension of $\pi_1(\Sigma_g)$ by a cyclic group and then use what we know about its group homology to show that this cyclic group is in fact trivial.

Summarizing what we know so far about $\pi_1(W)$, we have seen that it is a quotient of
\[ \left\langle \alpha_1,\dots,\alpha_g,\beta_1,\dots,\beta_g,\tau \mathrel{}\middle|\mathrel{} \prod_{i=1}^g [\alpha_i,\beta_i] = \tau^{2g-2}, [\alpha_i,\tau]=[\beta_i,\tau] = 1 \right\rangle, \]
where the central element $\tau$ is the image of a circle fiber $t \in \pi_1(Y_g)$.  Moreover, $H_1(W;\Z) = \Z^{2g}$ is generated by the elements $\alpha_i$ and $\beta_i$, and the central element $\tau$ belongs to the commutator subgroup of $\pi_1(W)$.  Thus there is a surjection
\[ p: \pi_1(\Sigma_g) = \left\langle \alpha_1,\dots,\alpha_g,\beta_1,\dots,\beta_g \mathrel{}\middle|\mathrel{} \prod_{i=1}^g [\alpha_i,\beta_i] = 1 \right\rangle \to \pi_1(W)/\langle\tau\rangle \]
through which the abelianization map $\ab: \pi_1(\Sigma_g) \to \Z^{2g}$ factors as
\[ \pi_1(\Sigma_g) \xrightarrow{p} \pi_1(W)/\langle\tau\rangle \xrightarrow{\ab} \Z^{2g}. \]
Such a factorization exists for any surjection $p: \pi_1(\Sigma_g) \to \pi_1(W)/\langle\tau\rangle$: since $\pi_1(\Sigma_g) \xrightarrow{\ab\circ p} \Z^{2g}$ is a map to an abelian group, it factors as $\pi_1(\Sigma_g) \xrightarrow{\ab} \Z^{2g} \xrightarrow{\psi} \Z^{2g}$, and $\psi$ is a surjection $\Z^{2g} \to \Z^{2g}$ since $\ab\circ p$ is onto, so it is an isomorphism and then $\ab: \pi_1(\Sigma_g) \to \Z^{2g}$ is equal to $(\psi^{-1}\circ\ab)\circ p$.

If we fix a surjection $\varphi: \Z^{2g} \to \Z/n\Z$ for some $n > 1$, then we have a collection of surjective maps of the form
\[ \xymatrix@!R=5pt{
\pi_1(W) \ar[dr] & & & \\
& \pi_1(W)/\langle\tau\rangle \ar[r]^-{\ab} & \Z^{2g} \ar[r]^-{\varphi} & \Z/n\Z \\
\pi_1(\Sigma_g) \ar[ur]_p & & &
} \]
and the kernels of the maps $\pi_1(W) \to \Z/n\Z$ and $\pi_1(\Sigma_g) \to \Z/n\Z$ define normal, $n$-fold cyclic covers $W'$ and $\Sigma_{g'}$ of $W$ and $\Sigma_g$ respectively, where $2-2g' = n(2-2g)$.

\begin{definition}
Let $(W,J)$ be a Stein filling of $(Y_g,\xi_g)$, and let $p: \pi_1(\Sigma_g) \to \pi_1(W)/\langle\tau\rangle$ be a surjection.  If $\Sigma_{g'} \to \Sigma_g$ and $W' \to W$ are the finite cyclic covers produced by the above construction, then we will say that $(\Sigma_{g'}, W')$ is \emph{induced by $(p,\varphi)$}.
\end{definition}

Since $W'$ is a finite cover of a Stein manifold it has a natural Stein structure $J'$, so its boundary $(Y',\xi')$ is connected, and as in the proof of Proposition~\ref{prop:stein-meridian} it follows that $Y'$ is a normal, $n$-fold cyclic cover of $Y_g$.

\begin{lemma} \label{lem:canonical-cyclic-cover}
If $(\Sigma_{g'},W')$ is induced by $(p,\varphi)$, then $(W',J')$ is a Stein filling of the canonical contact structure $(Y',\xi') = (Y_{g'},\xi_{g'})$ on the unit cotangent bundle of $\Sigma_{g'}$.
\end{lemma}

\begin{proof}
The circle fiber $t \in \pi_1(Y_g)$ is in the kernel of $\pi_1(Y_g) \xrightarrow{i_*} \pi_1(W) \to \Z/n\Z$ since it maps to $\tau \in \pi_1(W)$, so it lifts to a closed curve in $Y'$.  Its preimage in $Y'$ therefore consists of $n$ disjoint circles, so the orbit space $Y'/S^1$ is an $n$-fold cover of $\Sigma_g$, hence $Y'$ is a circle bundle over $\Sigma_{g'}$.  The Euler class of $Y'\to\Sigma_{g'}$ is then $n$ times the Euler class of $Y_g\to\Sigma_g$, namely $-n\chi(\Sigma_g) = -\chi(\Sigma_{g'})$, so in fact $Y'$ is the unit cotangent bundle $Y_{g'}$ of $\Sigma_{g'}$.

Since the contact structure $\xi_g$ is tangent to the fibers of $Y_g \to \Sigma_g$, its cover $\xi'$ is likewise tangent to the fibers of $Y_{g'} \to \Sigma_{g'}$, and the only contact structure on the unit cotangent bundle of $\Sigma_{g'}$ with Legendrian fibers is the canonical one \cite[Proposition~3.3]{giroux-circle-bundles} (cf.\ also \cite{lutz}).  Thus $(Y',\xi') = (Y_{g'}, \xi_{g'})$, and so $(W',J')$ is a Stein filling of $(Y_{g'},\xi_{g'})$.
\end{proof}

\begin{proposition} \label{prop:cover-surjection}
Suppose that $(\Sigma_{g'},W')$ is induced by $(p,\varphi)$.  Identifying $\pi_1(\Sigma_{g'})$ as a subgroup of $\pi_1(\Sigma_g)$, the map $p$ induces a surjection
\[ p': \pi_1(\Sigma_{g'}) \to \pi_1(W')/\langle\tau'\rangle \]
such that $\ker(p') = \ker(p)$.
\end{proposition}

\begin{proof}
It is clear that $\langle \tau \rangle \subset \pi_1(W')$, viewing the latter as a subgroup of $\pi_1(W)$, since $\tau$ is in the kernel of $\pi_1(W) \to \Z/n\Z$.  Moreover, if $\tau' \in \pi_1(W')$ denotes the image of the circle fiber $t' \in \pi_1(Y_{g'})$, then since $t'$ projects to the circle fiber $t \in \pi_1(Y_g)$, the covering map $W' \to W$ sends $\tau'$ to $\tau$, so we have
\[ \langle \tau \rangle \cap \pi_1(W') = \langle \tau' \rangle. \]
Thus the kernel of $\pi_1(W') \hookrightarrow \pi_1(W) \to \pi_1(W)/\langle \tau \rangle$ is $\langle \tau'\rangle$, inducing an injective map
\[ \frac{\pi_1(W')}{\langle \tau' \rangle} \hookrightarrow \frac{\pi_1(W)}{\langle \tau \rangle}, \]
and it follows that $\pi_1(W')/\langle \tau' \rangle$ has index $n$ in $\pi_1(W)/\langle \tau\rangle$.  Since $\pi_1(W')$ is by definition the kernel of the map $\pi_1(W) \to \Z/n\Z$, the group $\pi_1(W')/\langle\tau'\rangle$ sits in the kernel of the surjective $\pi_1(W)/\langle\tau\rangle \xrightarrow{\varphi\circ\ab} \Z/n\Z$, and this kernel has index $n$ so we conclude that
\[ \pi_1(W')/\langle\tau'\rangle = \ker( \varphi\circ\ab: \pi_1(W)/\langle\tau\rangle \to \Z/n\Z ). \]

Since $\pi_1(\Sigma_{g'})$ is the kernel of $(\varphi \circ \ab) \circ p$, it follows that $p(\pi_1(\Sigma_{g'}))$ lies in $\ker(\varphi\circ \ab)$, and so $p$ restricts to a map
\[ p': \pi_1(\Sigma_{g'}) \to \pi_1(W')/\langle \tau' \rangle, \]
which is easily seen to be surjective just as $p$ is.  Finally, since $p'$ is the restriction of $p$ to $\pi_1(\Sigma_{g'}) \subset \pi_1(\Sigma_{g})$ it follows that $\ker(p') = \ker(p) \cap \pi_1(\Sigma_{g'})$.  But $\ker(p) \subset \ker(\varphi\circ\ab\circ p) = \pi_1(\Sigma_{g'})$, and so $\ker(p)=\ker(p')$ as claimed.
\end{proof}

Proposition~\ref{prop:cover-surjection} allows us to characterize $\pi_1(W)$ as a cyclic extension of a surface group:

\begin{proposition}
\label{prop:central-extension}
The fundamental group $\pi_1(W)$ is a central extension of $\pi_1(\Sigma_g)$ by a cyclic group.  More precisely, there is a short exact sequence
\[ 1 \to \langle \tau \rangle \to \pi_1(W) \to \pi_1(\Sigma_g) \to 1 \]
with the image of $\langle \tau \rangle$ being central in $\pi_1(W)$.
\end{proposition}

\begin{proof}
It suffices to show that the surjection $p: \pi_1(\Sigma_g) \to \pi_1(W)/\langle\tau\rangle$ is also injective.  Supposing otherwise, let $x$ be a nontrivial element of $\ker(p)$.  Since surface groups are RFRS \cite{agol} (cf.\ also \cite{hempel}), there is a descending chain of subgroups
\[ \pi_1(\Sigma_g) = G_0 \supset G_1 \supset G_2 \supset \dots \]
such that each $G_{i+1}$ is a normal subgroup of $G_i$ with finite cyclic quotient, defined as the kernel of a map which factors through $G_i \to (G_i)^{\ab}$, and $\bigcap_{i=0}^\infty G_i = \{1\}$.  This corresponds to a tower of normal, finite cyclic covers
\[ \dots \to \Sigma_{g_2} \to \Sigma_{g_1} \to \Sigma_{g_0} = \Sigma_g \]
such that $\pi_1(\Sigma_{g_{i+1}}) = \ker\big(\pi_1(\Sigma_{g_i}) \xrightarrow{\ab} \Z^{2g_i} \xrightarrow{\varphi_i} \Z/n_i\Z\big)$ for some $\varphi_i$.  Now by induction, since $\ab: \pi_1(\Sigma_{g}) \to \Z^{2g}$ factors through $p_0 = p: \pi_1(\Sigma_g) \to \pi_1(W_0)/\langle\tau_0\rangle$, with $W_0=W$, we can construct for each $i \geq 0$ a normal cyclic cover $(W_{i+1},J_{i+1})$ of $(W_i,J_i)$ as above, with $(\Sigma_{g_{i+1}},W_{i+1})$ induced by $(p_i,\varphi_i)$.  
By Lemma~\ref{lem:canonical-cyclic-cover}, $(W_{i+1},J_{i+1})$ is a Stein filling of $(Y_{g_{i+1}}, \xi_{g_{i+1}})$, and Proposition~\ref{prop:cover-surjection} provides a surjection
\[ p_{i+1}: \pi_1(\Sigma_{g_{i+1}}) \to \pi_1(W_{i+1})/\langle\tau_{i+1}\rangle \]
with $\ker(p_{i+1})=\ker(p_i)$.  Thus $x \in \ker(p_0)$ implies that $x \in \ker(p_i) \subset G_i$ for all $i$.  But since $\bigcap G_i = \{1\}$ it follows that $x\not\in G_k$ for some $k \geq 0$, and this is a contradiction.
\end{proof}

\begin{proof}[Proof of Theorem~\ref{thm:surface-group}]
The circle fiber $\tau$ generates a $\Z/n\Z$ subgroup for some $n \geq 0$ (with $n=0$ if it is nontorsion), so Proposition~\ref{prop:central-extension} provides a short exact sequence of groups
\[ 1 \to \Z/n\Z \to \pi_1(W) \to \pi_1(\Sigma_g) \to 1 \]
for some $n \geq 0$.  We will show that $n=1$, and thus that $\pi_1(W)\to\pi_1(\Sigma_g)$ is an isomorphism.

The homologies of these groups with $\Z$ coefficients are related by the Lyndon/Hochschild-Serre spectral sequence (see e.g.\ \cite{brown}),
\[ E^2_{p,q} = H_p(\pi_1(\Sigma_g); H_q(\Z/n\Z; \Z)) \quad \Longrightarrow \quad H_{p+q}(\pi_1(W); \Z). \]
Since $\pi_1(\Sigma_g)$ has cohomological dimension 2, the $E^2$ page is supported in the interval $0\leq p \leq 2$.  Moreover, the homology of $\Z/n\Z$ is given by (letting $k \geq 1$):
\begin{align*}
H_q(\Z;\Z) &= \begin{cases} \Z, & q=0,1 \\ 0, & q\geq 2,\end{cases} &
H_q(\Z/k\Z;\Z) &= \begin{cases} \Z, & q=0 \\ \Z/k\Z, & q\mathrm{\ odd} \\ 0, & q\geq2 \mathrm{\ even}. \end{cases}
\end{align*}
In either case, the differential $d^2: E^2_{p,q} \to E^2_{p-2,q+1}$ must be identically zero with the possible exception of the map $\delta: E^2_{2,0} \to E^2_{0,1}$, since otherwise either the source or the target vanishes.  Each of the higher differentials $d^r: E^r_{p,q} \to E^r_{p-r,q+r-1}$ must vanish for $r \geq 3$ because either $p>2$ or $p-r<0$, so the spectral sequence collapses at the $E^3$ page, and we have
\begin{align*}
E^\infty_{0,2} &= 0, & E^\infty_{1,1} &= (\Z/n\Z)^{2g}, & E^\infty_{2,0} &= \ker(\delta: \Z \to \Z/n\Z).
\end{align*}
The convergence of this spectral sequence means that these are the associated graded groups of a filtration on $H_2(\pi_1(W);\Z)$.  But the latter group is $\Z$ by Proposition~\ref{prop:stein-hurewicz}, so each associated graded group must be cyclic, and since $E^\infty_{1,1}$ is cyclic we must have $n=1$.
\end{proof}

\subsection{The homotopy type of a Stein filling}

So far we have shown that if $(W,J)$ is a Stein filling of $(Y_g,\xi_g)$, then $W$ has the homology and intersection form of the disk cotangent bundle $DT^*\Sigma_g$ (Theorem~\ref{thm:homology-stein}) and that $\pi_1(W) \cong \pi_1(\Sigma_g)$ (Theorem~\ref{thm:surface-group}), with the circle fiber of $Y_g = \partial W$ being nullhomotopic in $W$.  In this section we will deduce that $W$ is therefore homotopy equivalent, and thus s-cobordant, to $DT^*\Sigma_g$ rel boundary.

\begin{proposition} \label{prop:tau-1-implies-aspherical}
If $(W,J)$ is a Stein filling of $(Y_g,\xi_g)$, then $W$ is aspherical.
\end{proposition}

\begin{proof}
A decomposition of $W$ into handles of index at most 2, with exactly one 0-handle, necessarily has $2g-1+k$ 1-handles and $k$ 2-handles for some $k\geq 1$ since $\chi(W) = 2-2g$.  The corresponding presentation of $\pi_1(W) \cong \pi_1(\Sigma_g)$ has $2g-1+k$ generators and $k$ relations and thus deficiency $2g-1$.

Hillman \cite[Proof of Theorem~2]{hillman-asphericity} showed that if a presentation $P$ of a group $G$ has deficiency $1 + \beta_1(G)$, where $\beta_1$ denotes the first $L^2$-Betti number (see for example \cite{luck}), then the 2-complex corresponding to $P$ is aspherical.  In the above case we know that $\beta_1(\pi_1(\Sigma_g)) = 2g-2$, so the 2-complex corresponding to the given presentation of $\pi_1(\Sigma_g)$ is aspherical, and thus $W$ (which retracts onto this complex) is aspherical as well.
\end{proof}

We have now shown that $W$ is a $K(\pi_1(\Sigma_g),1)$, and so it has the homotopy type of $DT^*\Sigma_g$.  Since both are compact 4-manifolds with boundary, we can strengthen this to an assertion about manifolds rel boundary as follows.

\begin{theorem} \label{thm:homotopy-equivalent}
If $(W,J)$ is a Stein filling of $(Y_g,\xi_g)$, then $W$ is s-cobordant rel boundary to the disk cotangent bundle $DT^*\Sigma_g$.
\end{theorem}

\begin{proof}
It suffices to find a homotopy equivalence $f: DT^*\Sigma_g \to W$ which restricts to a homeomorphism $\partial(DT^*\Sigma_g) \xrightarrow{\sim} \partial W$.  Since $W$ is compact and aspherical with $\pi_1(W)$ a surface group, Khan \cite[Corollary~1.23]{khan-decompositions} showed that $W$ is \emph{topologically s-rigid}, a condition which implies that if such an $f$ exists then $DT^*\Sigma_g$ is s-cobordant to $W$.

To construct $f: DT^*\Sigma_g \to W$, following Stipsicz \cite{stipsicz}, we first take a standard handlebody decomposition of $DT^*\Sigma_g$, with a 0-handle, $2g$ 1-handles, and a single 2-handle, and turn it upside down to build $DT^*\Sigma_g$ from a thickened $Y_g$ by attaching a 2-handle, $2g$ 3-handles, and a 4-handle.   We define $f$ by identifying the boundaries, $\partial (DT^*\Sigma_g) \xrightarrow{\sim} \partial W$, in a way which sends a circle fiber to a circle fiber; extending $f$ over the 2-handle of $DT^*\Sigma_g$, which can be done since the attaching curve is identified with the circle fiber in $Y_g = \partial W$ and is thus nullhomotopic in $W$; and then extending $f$ over the 3- and 4-handles of $DT^*\Sigma_g$, since the obstructions to doing so lie in $\pi_2(W)=0$ and $\pi_3(W)=0$.

The map $f$ which we have constructed now induces an isomorphism $f_*: \pi_1(DT^*\Sigma_g) \to \pi_1(W)$, since it induces an isomorphism $\pi_1(\partial(DT^*\Sigma_g)) \xrightarrow{\sim} \pi_1(\partial W)$ which preserves the subgroup generated by the circle fiber, and both groups are quotients of $\pi_1(Y_g)$ by that subgroup.  Moreover, $f$ induces an isomorphism on all higher homotopy groups, since these are identically zero, and so $f$ is a homotopy equivalence by Whitehead's theorem.
\end{proof}

\bibliographystyle{alpha}
\bibliography{References}

\begin{thebibliography}{PVHM10}

\bibitem[Ago08]{agol}
Ian Agol.
\newblock Criteria for virtual fibering.
\newblock {\em J. Topol.}, 1(2):269--284, 2008.

\bibitem[Bau08]{bauer}
Stefan Bauer.
\newblock Almost complex 4-manifolds with vanishing first {C}hern class.
\newblock {\em J. Differential Geom.}, 79(1):25--32, 2008.

\bibitem[Bro82]{brown}
Kenneth~S. Brown.
\newblock {\em Cohomology of groups}, volume~87 of {\em Graduate Texts in
  Mathematics}.
\newblock Springer-Verlag, New York-Berlin, 1982.

\bibitem[Etn04]{etnyre}
John~B. Etnyre.
\newblock Planar open book decompositions and contact structures.
\newblock {\em Int. Math. Res. Not.}, (79):4255--4267, 2004.

\bibitem[FQ90]{freedman-quinn}
Michael~H. Freedman and Frank Quinn.
\newblock {\em Topology of 4-manifolds}, volume~39 of {\em Princeton
  Mathematical Series}.
\newblock Princeton University Press, Princeton, NJ, 1990.

\bibitem[Fre82]{freedman}
Michael~Hartley Freedman.
\newblock The topology of four-dimensional manifolds.
\newblock {\em J. Differential Geom.}, 17(3):357--453, 1982.

\bibitem[FS95]{fs-blowup}
Ronald Fintushel and Ronald~J. Stern.
\newblock Immersed spheres in {$4$}-manifolds and the immersed {T}hom
  conjecture.
\newblock {\em Turkish J. Math.}, 19(2):145--157, 1995.

\bibitem[Gir01]{giroux-circle-bundles}
Emmanuel Giroux.
\newblock Structures de contact sur les vari\'et\'es fibr\'ees en cercles
  au-dessus d'une surface.
\newblock {\em Comment. Math. Helv.}, 76(2):218--262, 2001.

\bibitem[Hem72]{hempel}
John Hempel.
\newblock Residual finiteness of surface groups.
\newblock {\em Proc. Amer. Math. Soc.}, 32:323, 1972.

\bibitem[Hil97]{hillman-asphericity}
Jonathan~A. Hillman.
\newblock On {$L^2$}-homology and asphericity.
\newblock {\em Israel J. Math.}, 99:271--283, 1997.

\bibitem[Hin00]{hind-rp3}
Richard Hind.
\newblock Holomorphic filling of {${\bf R}{\rm P}^3$}.
\newblock {\em Commun. Contemp. Math.}, 2(3):349--363, 2000.

\bibitem[Kal13]{kaloti}
Amey Kaloti.
\newblock Stein fillings of planar open books.
\newblock arXiv:1311.0208, 2013.

\bibitem[Kha12]{khan-decompositions}
Qayum Khan.
\newblock Homotopy invariance of 4-manifold decompositions: connected sums.
\newblock {\em Topology Appl.}, 159(16):3432--3444, 2012.

\bibitem[Li06a]{li-quaternionic}
Tian-Jun Li.
\newblock Quaternionic bundles and {B}etti numbers of symplectic 4-manifolds
  with {K}odaira dimension zero.
\newblock {\em Int. Math. Res. Not.}, pages Art. ID 37385, 28, 2006.

\bibitem[Li06b]{li-kodaira-zero}
Tian-Jun Li.
\newblock Symplectic 4-manifolds with {K}odaira dimension zero.
\newblock {\em J. Differential Geom.}, 74(2):321--352, 2006.

\bibitem[Lis08]{lisca-lens}
Paolo Lisca.
\newblock On symplectic fillings of lens spaces.
\newblock {\em Trans. Amer. Math. Soc.}, 360(2):765--799 (electronic), 2008.

\bibitem[LMY14]{lmy}
Tian-Jun Li, Cheuk~Yu Mak, and Kouichi Yasui.
\newblock Uniruled caps and {C}alabi-{Y}au caps.
\newblock arXiv:1412.3208, 2014.

\bibitem[LS91]{lalonde-sikorav}
Fran{\c{c}}ois Lalonde and Jean-Claude Sikorav.
\newblock Sous-vari\'et\'es lagrangiennes et lagrangiennes exactes des fibr\'es
  cotangents.
\newblock {\em Comment. Math. Helv.}, 66(1):18--33, 1991.

\bibitem[L{\"u}c02]{luck}
Wolfgang L{\"u}ck.
\newblock {\em {$L^2$}-invariants: theory and applications to geometry and
  {$K$}-theory}, volume~44 of {\em Ergebnisse der Mathematik und ihrer
  Grenzgebiete. 3. Folge. A Series of Modern Surveys in Mathematics [Results in
  Mathematics and Related Areas. 3rd Series. A Series of Modern Surveys in
  Mathematics]}.
\newblock Springer-Verlag, Berlin, 2002.

\bibitem[Lut83]{lutz}
Robert Lutz.
\newblock Structures de contact et syst\`emes de {P}faff \`a pivot.
\newblock In {\em Third {S}chnepfenried geometry conference, {V}ol. 1
  ({S}chnepfenried, 1982)}, volume 107 of {\em Ast\'erisque}, pages 175--187.
  Soc. Math. France, Paris, 1983.

\bibitem[McD90]{mcduff-rational-ruled}
Dusa McDuff.
\newblock The structure of rational and ruled symplectic {$4$}-manifolds.
\newblock {\em J. Amer. Math. Soc.}, 3(3):679--712, 1990.

\bibitem[McD91]{mcduff-disconnected}
Dusa McDuff.
\newblock Symplectic manifolds with contact type boundaries.
\newblock {\em Invent. Math.}, 103(3):651--671, 1991.

\bibitem[MS97]{morgan-szabo}
John~W. Morgan and Zolt{\'a}n Szab{\'o}.
\newblock Homotopy {$K3$} surfaces and mod {$2$} {S}eiberg-{W}itten invariants.
\newblock {\em Math. Res. Lett.}, 4(1):17--21, 1997.

\bibitem[OO05]{ohta-ono}
Hiroshi Ohta and Kaoru Ono.
\newblock Simple singularities and symplectic fillings.
\newblock {\em J. Differential Geom.}, 69(1):1--42, 2005.

\bibitem[Pol91]{polterovich}
L.~Polterovich.
\newblock The surgery of {L}agrange submanifolds.
\newblock {\em Geom. Funct. Anal.}, 1(2):198--210, 1991.

\bibitem[PVHM10]{pvhm}
Olga Plamenevskaya and Jeremy Van Horn-Morris.
\newblock Planar open books, monodromy factorizations and symplectic fillings.
\newblock {\em Geom. Topol.}, 14(4):2077--2101, 2010.

\bibitem[Sei08]{seidel}
Paul Seidel.
\newblock {\em Fukaya categories and {P}icard-{L}efschetz theory}.
\newblock Zurich Lectures in Advanced Mathematics. European Mathematical
  Society (EMS), Z\"urich, 2008.

\bibitem[Sta65]{stallings-homology}
John Stallings.
\newblock Homology and central series of groups.
\newblock {\em J. Algebra}, 2:170--181, 1965.

\bibitem[Sta15]{starkston}
Laura Starkston.
\newblock Symplectic fillings of {S}eifert fibered spaces.
\newblock {\em Trans. Amer. Math. Soc.}, 367(8):5971--6016, 2015.

\bibitem[Sti02]{stipsicz}
Andr{\'a}s~I. Stipsicz.
\newblock Gauge theory and {S}tein fillings of certain 3-manifolds.
\newblock {\em Turkish J. Math.}, 26(1):115--130, 2002.

\bibitem[Tau95]{taubes-more}
Clifford~Henry Taubes.
\newblock More constraints on symplectic forms from {S}eiberg-{W}itten
  invariants.
\newblock {\em Math. Res. Lett.}, 2(1):9--13, 1995.

\bibitem[Tau96]{taubes-sw-gr}
Clifford~H. Taubes.
\newblock {${\rm SW}\Rightarrow{\rm Gr}$}: from the {S}eiberg-{W}itten
  equations to pseudo-holomorphic curves.
\newblock {\em J. Amer. Math. Soc.}, 9(3):845--918, 1996.

\bibitem[Wen10]{wendl}
Chris Wendl.
\newblock Strongly fillable contact manifolds and {$J$}-holomorphic foliations.
\newblock {\em Duke Math. J.}, 151(3):337--384, 2010.

\bibitem[Wen14]{wendl-blog}
Chris Wendl.
\newblock A biased survey on symplectic fillings, part 7 (maximal elements and
  co-fillability).
\newblock
  \url{https://symplecticfieldtheorist.wordpress.com/2014/12/29/a-biased-survey-on-symplectic-fillings-part-7-maximal-elements-and-co-fillability/},
  2014.

\end{thebibliography}

\end{document}